\documentclass[12pt]{amsart}
\advance\hoffset by -0.5 truecm
\usepackage{amssymb}
\newtheorem{Theorem}{Theorem}[section]
\newtheorem{Lemma}[Theorem]{Lemma}

\def\V{\mbox{Var}}

\def\R\re
\def\V{\bf V}

\def \re{{\mathbb R}}

\def \0{\lambda_{0}}

\begin{document}
\title[Note Yamabe equation]{A note on solutions of Yamabe-type  equations on products of spheres}

\author{Jimmy Petean}\thanks{J. Petean is supported by grant 220074 of Fondo Sectorial de Investigaci\'{o}n para la Educaci\'{o}n SEP-CONACYT}
\address{Centro de Investigaci\'{o}n en Matem\'{a}ticas, CIMAT, Calle Jalisco s/n, 36023 Guanajuato, Guanajuato, M\'{e}xico}
\email{jimmy@cimat.mx}%

\author{H\'ector Barrantes G. }
\address{Centro de Investigaci\'{o}n en Matem\'{a}ticas, CIMAT, Calle Jalisco s/n, 36023 Guanajuato, Guanajuato, M\'{e}xico.
Universidad de Costa Rica. Sede de Occidente. 20201. Alajuela. Costa Rica.}
\email{hector.barrantes@cimat.mx, hector.barrantes@ucr.ac.cr}%

\begin{abstract}We consider  Yamabe-type  equations on the Riemannian product of constant curvature metrics on $\textbf{S}^n \times\textbf{ S}^n$,
and study solutions  which are invariant by the cohomogeneity one diagonal action of $O(n+1)$. We obtain multiplicity results for both positive
and nodal solutions. In particular we prove the existence of
nodal solutions of the Yamabe equation on these products  which depend non-trivially on both factors.

\end{abstract}
\maketitle

\section{Introduction}

Given a Riemannian manifold $(M^n , g)$ of dimension $n\geq 3$,  the {\it Yamabe equation} is

$$- a_n \Delta_g u + s_g u = \lambda |u|^{p_n -2} u,$$

\noindent
where $s_g$ is the scalar curvature of $g$, $a_n =\frac{4(n-1)}{n-2}$, $p_n =\frac{2n}{n-2}$ and $\lambda \in \re$. In
case $u$ is a positive solution of this equation then the conformal metric $u^{p_n -2} g$ has constant scalar curvature
$\lambda$. An important role in the study of the equation has been played by the {\it Yamabe constant} of the conformal class
$[g]$ of $g$, which we denote by $Y(M,[g])$ and is given by

$$Y(M,[g])= \inf_{h\in [g]} \frac{\int_M s_h  \ dv_h}{(Vol(M,h)^{\frac{n-2}{n}}},$$

\noindent
where $dv_h$ is the volume element of the metric $h$ and $Vol(M,h)$ is the volume of $(M,h)$.
Throughout the combined efforts of H. Yamabe \cite{Yamabe}, N. Trudinger \cite{Trudinger}, T. Aubin \cite{Aubin} and R. Schoen
\cite{Schoen} it was proved that the equation
always has at least one solution, for which the corresponding constant scalar curvature metric realizes the
Yamabe constant. The equation can be normalized so that $\lambda$ is $-1, 0$ or $1$, according to  the sign of
$Y(M,g)$. The minimizing solution is the unique solution in case $Y(M,[g]) \leq 0$, i.e. $\lambda=-1$ or $\lambda =0$.
It is also unique in case the minimizing solution is
Einstein and different from the constant curvature metric on ${\bf S}^n$, which we will denote by $g_0^n$,
by a theorem of M. Obata \cite{Obata}. $({\bf S}^n , g_0^n )$
has a non-compact group of conformal transformations, which give a non-compact family of solutions to the Yamabe equation.
This is the first example of multiplicity of solutions and plays an important role in the theory. Another important example, or family of
examples, where one has multiplicity of solutions is that of Riemannian products.  If $(M \times N, g+h)$ is a Riemannian product
with constant scalar curvature, with $s_h >0$, then one considers $(M^n \times N^k , g + \delta h )$ for $\delta >0$ and small.
It can be seen by general considerations that for $\delta$ small enough the product metric cannot be a minimizer for the Yamabe
constant. Therefore there must be at least one other solution. But there are several results showing that the number of
solutions grows as $\delta \rightarrow 0$, see for instance
\cite{BP2, Piccione, Henry, Schoen2}. The solutions built in these articles are actually functions of $M$.
A function
$u:M \rightarrow \re$, considered as a function on $M\times N$ gives a solution for the Yamabe equation for $g+\delta h$ if it
satisfies

$$-a_{n+k} \Delta_g u + (s_g + \delta^{-1} s_h )u= |u|^{p_{n+k} -2} u ,$$

\noindent
where we have normalized the positive constant $\lambda$ to be 1. Note that $p_{n+k} < p_n$. So in this context one would
be interested in positive solutions of the {\it subcritical} equation

$$-\Delta u + \lambda u = \lambda |u|^{p-2} u ,$$

\noindent
where $\lambda$ is a positive constant, $2<p\leq p_n$, and we have now renormalized the equation so that $u\equiv 1$ is a solution.

Many multiplicity results have also  been obtained for these equations
using Lyapunov-Schmidt reduction and other topological methods (see for instance \cite{Dancer, DKM, Micheletti}).

There has also been interest in {\it nodal} solutions of the equations (i.e. solutions that change sign). See for instance the articles
\cite{Ammann2, Clapp1, Clapp2, Ding, JC, GHenry, Robert} and the references in them. Nodal solutions $u$ do
 not give metrics of constant scalar curvature since $u$ vanishes at some points and therefore $|u|^{p_n -2} g$ is
not
a Riemannian metric. But they have geometric interest. The existence of at least one nodal solution is proved
in general cases in \cite{Ammann2}, as minimizers for the second Yamabe invariant. But there are not as many results
about multiplicity of nodal solutions as in the positive case.

\vspace{.5cm}

In this article we will  consider the products of spheres $({\bf S}^n \times {\bf S}^n , g_0^n \times \delta g_0^n )$, with  $\delta >0$
(note that for $\delta =1$ the metric is Einstein). Solutions for $\delta$ small and subcritical exponent
 have been built in \cite{Henry, YanYan, Petean2}, which depend only on the first factor.  Interest in finding all solutions of the Yamabe equation  in
this case comes from trying to compute the Yamabe constants and its limit $\lim_{\delta \rightarrow 0}Y({\bf S}^n \times {\bf S}^n , [g_0^n +
\delta g_0^n ]) = Y({\bf S}^n \times \re^n , g_0^n + dx^2)$ (see  \cite{Akutagawa}).
An important question raised in \cite{Akutagawa, Ammann}, related to the  computations of these Yamabe constants, is whether all solutions of the
Yamabe equation on certain Riemannian products, like products of spheres (or the product of a sphere with Euclidean space),
 depend on only one of the factors. The main goal of this article
is to show the existence of solutions which depend non-trivially on both factors: positive solutions  when $p<p_n$  and nodal solutions when $p=p_n$. To build such solutions we consider
the isometric $O(n+1)$-action on ${\bf S}^n \times {\bf S}^n$ given by $A\cdot (x,y) =(Ax, A y )$. It is a cohomogeneity one action. From now
on by an {\it invariant function} on ${\bf S}^n \times {\bf S}^n$ we will mean a function which is invariant by this diagonal $O(n+1)$-action, there is
no risk of confusion since we will only consider this action. Note that any invariant function which is not constant depends non-trivially
on both factors of ${\bf S}^n \times {\bf S}^n$.

We first consider nodal solutions of the (critical) Yamabe equation on the products $({\bf S}^n \times {\bf S}^n , g_0^n +
\delta g_0^n )$. We will prove:

\begin{Theorem}\label{nodal} For any $\delta >0$ the Yamabe equation on $({\bf S}^n \times {\bf S}^n , g_0^n + \delta g_0^n )$,

\begin{equation}
-a_{2n} \Delta_{g_0^n + \delta g_0^n} u + n(n-1) (1+\delta^{-1} ) u = n(n-1) (1+\delta^{-1} ) |u|^{p_{2n} -2} u
\end{equation}

\noindent
admits infinite nodal solutions which are invariant by the diagonal action of $O(n+1)$.
\end{Theorem}

We will also prove that the number of positve solutions of the subcritical equation grows as  $\lambda \rightarrow \infty$. As
we mentioned before this gives multiplicity results for the Yamabe equation on certain Riemannian products. We will prove:

\begin{Theorem}\label{positive-solutions-theorem}
 For each $\delta >0$, $p \in (2, p_{2n} )$,  let $\lambda_{k,\delta, p} := \frac{k(k+n-1)}{p-2}\left(1+ \frac{1}{\delta}\right)$.
If $\lambda \in \bigl(\lambda_{k,\delta, p}, \lambda_{k+1,\delta, p}\bigr]$ then
the subcritical equation  on $({\bf S}^n \times {\bf S}^n , g_0^n \times \delta g_0^n )$,

\begin{equation}\label{positive}
- \Delta_{g_0^n + \delta g_0^n} u + \lambda u =\lambda u^{p-1} ,
\end{equation}

\noindent
has at least $k$ positive solutions which are invariant by the diagonal action of $O(n+1)$.
\end{Theorem}

In Section 2 we will discuss the setting of Yamabe-type equations restricted to the space of functions invariant
by the diagonal action of $O(n+1)$ on ${\bf S}^n \times {\bf S}^n$. We will discuss nodal solutions of the Yamabe equation and prove
Thoerem 1.1 in Section 3. Finally in Section 4 we will consider positive solutions of subcritical equations and
prove Theorem 1.2.

\section{Yamabe-type equations for invariant functions}

We consider $n\geq 2$ and let $g_0^n$ denote the curvature 1 metric on ${\bf S}^n$.
For any $\delta >0$ we consider the Riemannian product  $ G_{\delta}=g_0^n +\delta g_0^n$ and the isometric action
of $O(n+1)$ on $({\bf S}^n \times {\bf S}^n , G_{\delta} )$ given by $A\cdot(x,y)=(Ax,Ay)$.

Let $f \colon {\bf S}^n \times {\bf S}^n \to [-1,1]$ be given by:
$$f(p, q) = \langle p, q \rangle.$$

Note that $f$ is invariant by the action of $O(n+1)$. By a direct computation we obtain

$$\Delta_{G_{\delta}}f = -n\left( 1+ \frac{1}{\delta}\right) f , \;\;\;\;\;\;
   \left|\nabla_{G_\delta}f\right|_{G_\delta}^2 =\left( 1+ \frac{1}{\delta}\right)(1-f^2).$$

This implies that $f$ is  an isoparametric function (see \cite{Wang} for the definition and basic results concerning isoparametric functions). The only critical values of $f$ are its minimum -1 and its maximum 1. Every invariant
function $u\colon {\bf S}^n \times {\bf S}^n \to \re$ can be written as $u = \varphi \circ f$, where $\varphi \colon [-1,1] \to \re$.
Since $f$ is smooth the regularity of $u$ is equal to the regularity of $\varphi$. We have that

$$\Delta_{G_\delta}u = ( \varphi '' \circ f ) . |\nabla_{G_{\delta}} f |^2_{G_{\delta}} + (\varphi ' \circ f) .  \Delta_{G_{\delta}} f  .$$

$$= \left[ \left( 1+ \frac{1}{\delta}\right)(1-t^2)  \   \varphi '' -n\left( 1+ \frac{1}{\delta}\right) t  \  \varphi ' \right] \circ f $$

\noindent
for $t\in [-1,1]$.

Therefore $u$ solves

\begin{equation}\label{ecuaciongeneral-solucionesnodales-SnxSn}
-\Delta_{G_{\delta}} u + \lambda u  = \lambda |u|^{p-2}u \;\;\mbox{on }\;\; {\bf S}^n \times {\bf S}^n
\end{equation}

\noindent
if and only if

\begin{equation}\label{t}
-(1-t^2)\varphi ''(t)+nt\varphi '(t)+\frac{\lambda }{ 1+ \frac{1}{\delta} }\varphi (t)=\frac{\lambda}{ 1+ \frac{1}{\delta}} |\varphi |^{p-2}\varphi .
\end{equation}

If we now call  $w(r) = \varphi (\cos (r))$  then $w'(0) = w'(\pi ) =0$ and $\varphi$ solves   equation (\ref{t})
if and only if

\begin{equation}\label{w(r)}
w''(r)+(n-1)\frac{\cos (r)}{\sin (r)}w'(r) + \frac{\lambda}{ 1+ \frac{1}{\delta}} \left( |w(r)|^{p-2}w(r) -w(r)\right)=0
\end{equation}
with $r \in[0,\pi]$.

For any  $\alpha >0$  we call  $w_{\alpha}: [0,\pi ) \rightarrow \re$  the solution of (\ref{w(r)}) with initial conditions
$$w_{\alpha}(0) =\alpha, \;\;\;\;\;w_{\alpha}'(0) =0.$$

If $w_{\alpha}$ extends up to the singularity at $\pi$ and  $w_{\alpha}'(\pi )=0$, then $\varphi_{\alpha}  (t) = w_{\alpha} (\arccos (t))$ is a $C^2$ function which solves equation
(\ref{t}).
Then $u=\varphi_{\alpha} \circ f$ solves equation (\ref{ecuaciongeneral-solucionesnodales-SnxSn}).

\vspace{.5cm}

When   $\lambda = \frac{s_{G_{\delta}} }{a_{2n}}= \frac{n(n-1)(2n-2)(1+\frac{1}{\delta})}{4(2n-1)} $ and  $ p=\frac{4n}{2n-2}$ equation (\ref{ecuaciongeneral-solucionesnodales-SnxSn}) is the
Yamabe equation for $({\bf S}^n \times {\bf S}^n , G_{\delta} )$. Therefore Theorem \ref{nodal} follows from the following:

\begin{Theorem}Let $p=\frac{4n}{2n-2}$. For any $\lambda >0$ and any integer $k$ there exists $\alpha_k >0$ such that $w_{\alpha_k}'(\pi )=0$
and $w_{\alpha_k}$ has exactly $k$ zeroes on $(0,\pi )$.

\end{Theorem}

Similarly Theorem \ref{positive-solutions-theorem} follows from the next theorem. For any $\delta >0$ and
any $p\in (2,p_{2n} )$ we call $\lambda_k =\frac{k(k+n-1)}{p-2} (1+\delta^{-1})$. Then we have:

\begin{Theorem} For any $p\in (2,p_{2n} )$, $\delta  >0$ and $\lambda \in (\lambda_k ,\lambda_{k+1} ]$ there exist
at least $k$ positive different solutions of
equation (\ref{w(r)}) verifying the boundary conditions  $w'(0)=w ' (\pi )=0$.

\end{Theorem}

Theorem 2.1 will be proved in Section 3 and Theorem 2.2 will be proved in Section 4. To finish this section we introduce
the energy functional. Let $\mu= \frac{\lambda}{ 1+ \frac{1}{\delta}}$ and

$$E_{\alpha}(r):=\frac{1}{2}(w'_{\alpha}(r))^2+ \mu \left(\frac{|w_{\alpha}(r)|^p}{p}-\frac{w^2_{\alpha}(r)}{2}\right)$$

Note that

$$
E'_{\alpha}(r) = -(n-1)\frac{\cos (r)}{\sin (r)} (w_{\alpha } '(t))^2  ,
$$

\noindent
therefore $E_{\alpha}$ is decreasing in $(0,\pi/2)$ and increasing in $(\pi /2 ,\pi )$.

Note also that if $w_{\alpha} (r_0 )=0$ then $E_{\alpha}(r_0 ) =\frac{1}{2} (w_{\alpha}'(r_0))^2 \geq 0$
(and the equality holds if and only if $\alpha =0$). For instance this implies

\begin{Lemma} \label{lema-energia} If  $\alpha >0$ is such that  $E_{\alpha}(0)\leq 0$ then
 $w_{\alpha}(r)>0$, for all  $r \in (0,\frac{\pi}{2})$.
\end{Lemma}

Note that
$$
E_{\alpha}(0) = \frac{1}{2}(w'_{\alpha}(0))^2+\mu \left(\frac{|w_{\alpha}(0)|^p}{p}-\frac{w^2_{\alpha}(0)}{2}\right)
= \mu \left(\frac{\alpha^p}{p}-\frac{\alpha^2}{2}\right)
=\mu \frac{\alpha^2}{p}\left(\alpha^{p-2} -\frac{p}{2}\right)
$$

\noindent
and the lemma says that if $w_{\alpha}$ has a zero in  $(0,\frac{\pi}{2})$, then $\alpha > \left(\frac{p}{2}\right)^{\frac{1}{p-2}}$.

\section{Proof of Theorem 2.1}

We fix $p=p_{2n}$ in equation (\ref{w(r)}).
We begin with some elementary lemmas concerning equation (\ref{w(r)}). We call
$\mu= \frac{\lambda}{ 1+ \frac{1}{\delta}}$.

\begin{Lemma} \label{w-simetrica}
Let  $\alpha > 0$ be such that the solution   $w_{\alpha}$ of (\ref{w(r)}) satisfies $ w_{\alpha}'\left(\frac{\pi}{2}\right) =0$
then $w_{\alpha} (\pi  -t)= w_{\alpha} (t)$  for all $t\in [0, \pi )$ and therefore $w_{\alpha} '(\pi) =0$.
\end{Lemma}

\begin{proof}
Let  $h(t) = w_{\alpha} (\pi-t) $ with $t \in [0,  \pi)$. Note that $h$ is a solution of equation (\ref{w(r)}).
Moreover
$h\left( \frac{\pi}{2}\right) = w_{\alpha} \left( \frac{\pi}{2}\right)$,  $ h'\left( \frac{\pi}{2}\right) = w_{\alpha}'\left( \frac{\pi}{2}\right) =0$.
By the uniqueness of solutions $h=w_{\alpha}$.
Therefore $w_{\alpha} (\pi - t )=h(t) =w_{\alpha} (t)$, proving the lemma.
\end{proof}

\begin{Lemma} \label{w-antisimetrica}
Let $\alpha > 0$ be such that the solution  $w_{\alpha}$ de (\ref{w(r)}) satisfies $ w_{\alpha}\left(\frac{\pi}{2}\right) =0$.
Then $w_{\alpha}(\pi  -t)= -w_{\alpha} (t)$  for all $t\in [0, \pi )$ and therefore $w_{\alpha} '(\pi) =0$.
\end{Lemma}

\begin{proof}
Let  $h(t) = -w_{\alpha} (\pi-t) $ with $t \in [0, \pi)$. Note that $h$ is a solution of equation  (\ref{w(r)}).
Since
$h\left( \frac{\pi}{2}\right) = w_{\alpha} \left( \frac{\pi}{2}\right) =0$, $ h'\left( \frac{\pi}{2}\right) = w_{\alpha}'\left( \frac{\pi}{2}\right) $
it follows from the uniqueness of solutions that $h=w_{\alpha}$, proving the lemma.

\end{proof}

\begin{Lemma}{\label{lema-numerodeceros-w-alpha1}}
Let $\alpha_0 > 1$ be such that the solution  $w_{\alpha_0}$ of  (\ref{w(r)}) has
exactly $k$ zeroes in  $\left(0, \frac{\pi}{2}\right)$ and
 $ w_{\alpha_0}\left(\frac{\pi}{2}\right) \neq 0$.  Then there exists $\varepsilon > 0$ such that for any
 $\alpha \in (\alpha_0 - \varepsilon , \alpha_0 +\varepsilon   )$ the solution $w_{\alpha}$
 has exactly  $k $ zeroes in  $\left(0, \frac{\pi}{2}\right)$ and $w_{\alpha} (\pi /2 ) \neq 0$.
\end{Lemma}

\begin{proof} Let $0< z_1 <...,z_k <\pi /2$ be the k zeroes of $w_{\alpha_0}$ in $(0,\pi /2 )$. Let $\delta >0$ be
small enough so that $w_{\alpha_0} ' (t) \neq 0$ for any $i=1,...,k$ and any $t\in [z_i -\delta , z_i + \delta ]$. For
$\varepsilon >0$ small enough we can assume that for any $\alpha \in (\alpha_0 - \varepsilon , \alpha_0 +\varepsilon   )$
we have that $w_{\alpha} >0 $ in $[0,z_1 -\delta ]$, $w_{\alpha}'  <0 $ in $[z_1 - \delta , z_1 + \delta ]$,
$w_{\alpha} <0 $ in $[z_1 + \delta ,z_2 -\delta ]$, $w_{\alpha}'  >0 $ in $[z_2 - \delta , z_2 + \delta ]$, and so on.
It follows that $w_{\alpha}$ has exactly $k$ zeroes in $(0,\pi /2 )$.
\end{proof}

\begin{Lemma}{\label{lema-numerodeceros-w-alpha2}}
 Let $\alpha_0 > 1$ be such that the solution $w_{\alpha_0}$ of equation  (\ref{w(r)}) has
exactly  $k$ zeroes in $\left(0, \frac{\pi}{2}\right)$ and
 $ w_{\alpha_0}\left(\frac{\pi}{2}\right) =0$.  Then there exists  $\varepsilon > 0$ such that for any
 $\alpha \in (\alpha_0 - \varepsilon , \alpha_0 +\varepsilon   )$ the solution $w_{\alpha}$
  has either exactly $k $ zeroes or exactly  $k+1$ zeroes in  $\left(0, \frac{\pi}{2}\right)$.
\end{Lemma}

\begin{proof} Since $ w_{\alpha_0}\left(\frac{\pi}{2}\right) =0$ and $\alpha_0 \neq 0$ we have that
$ w_{\alpha_0} ' \left(\frac{\pi}{2}\right) \neq 0$. Let $\delta >0$ be small enough so that
$ w_{\alpha_0} ' (t) \neq 0$ if $t\in [\pi / 2 -\delta , \pi /2 + \delta]$. Choose $\varepsilon >0$ small enough
so that  for any $\alpha \in (\alpha_0 -\varepsilon , \alpha_0 + \varepsilon )$ $ w_{\alpha} ' (t) \neq 0$ for any
$t\in [\pi / 2 -\delta , \pi /2 + \delta]$ and $w_{\alpha}$ has exactly one zero in $[\pi / 2 -\delta , \pi /2 + \delta]$.
By the same argument as in the previous lemma we can also assume that
$\varepsilon$ is small enough so that for any  $\alpha \in (\alpha_0 -\varepsilon , \alpha_0 + \varepsilon )$ the
solution $w_{\alpha}$ has exactly $k$ zeroes in $[0, \pi /2 - \delta ]$. Therefore $w_{\alpha}$ has either $k$ or $k+1$ in
$(0, \pi /2 )$, depending on whether  its zero in $[\pi /2 -\delta , \pi /2 + \delta ]$ is $< \pi /2$ or not.
\end{proof}

Now we are ready to prove Theorem 2.1.

\begin{proof}Note that we are considering equation (\ref{w(r)}) with $p=p_{2n}<p_n$.
Consider an integer  $i>>k$. It then follows from  \cite[Theorem 3.1]{JC} that there exists  $\alpha_* >1 $
such that $w_{\alpha_*}$ has at least  $i$ zeroes in  $\left(0, \frac{\pi}{2}\right)$.

First consider the set
$$A_0 := \left\{ \alpha \in (1, \alpha_* ] :  \;\; w_{\alpha} \geq 0\; \;\mbox{in}\; \; \left(0,\frac{\pi}{2}\right] \right\}.$$

Note that by Lemma \ref{lema-energia}  $(1,(\frac{p}{2} )^{\frac{1}{p-2}} ] \subset A_0$ and that $A_0$
is closed in  $(1, \alpha_* ]$.
Let $a_0 := \sup A_0$. If $t\in (0, \pi /2 )$ and $w_{a_0} (t)=0$ then $t$ would be a local minimum for $w_{a_0}$ and
therefore $w_{a_0} ' (t) =0$. By uniqueness we would have that $w_{a_0 } \equiv 0$, which is a contradiction. Then $w_{a_0}$
is strictly positive in $[0,\pi /2 )$. Moreover $w_{a_0} \left(\frac{\pi}{2}\right) =0$, since
$a_0$ is the supremum of $A_0$. Therefore by Lemma 3.2 $w_{a_0} ' (\pi ) =0$ and $w_{a_0}$ has exactly one zero in
$[0,\pi ]$.

Now define
$$A_1 := \left\{ \alpha \in (1, \alpha_* ] : \;\; w_{\alpha}  \; \;\mbox{has exactly one zero in }\; \; \left(0,\frac{\pi}{2}\right) \right\}$$

If $a_0 < \alpha \leq \alpha_* $ then $w_{\alpha} $ has at least one zero in  $\left( 0, \frac{\pi}{2}\right)$. Therefore by Lemma 3.4
there exists $a>a_0$ such that
$w_{a} $ has exactly one zero in  $\left( 0, \frac{\pi}{2}\right)$.
Therefore
$A_1 \neq \emptyset$ and it is bounded.
Let $a_1 := \sup A_1$.
Note that $a_1 > a_0$. By Lemma 3.3 $w_{a_1} (\pi /2 )=0$. Since $a_1 >a_0$ it follows that
$w_{a_1}$ has exactly one zero in  $\left(0, \frac{\pi}{2}\right)$.

Now for any $j\geq 2$, $j<i$  define

$$A_j := \left\{ \alpha \in (1, \alpha_* ] : \;\; w_{\alpha}  \; \;\mbox{has exactly j zeroes  in }\; \; \left(0,\frac{\pi}{2}\right) \right\}$$

Assume that $A_j \neq \emptyset$ and $a_j = \sup A_j > a_{j-1} >a_{j-2} >...>a_1$.
If $j+1 <i$ then it follows from Lemma \ref{lema-numerodeceros-w-alpha1}  and Lemma 3.4 that $w_{a_j} (\pi /2) =0$, $A_{j+1} \neq \emptyset$.  As in the case $j=1$ we see then that
$a_{j+1}=\sup A_{j+1} >a_j$.
By induction we see that $\forall j \geq 2$, $A_j \neq \emptyset$ and $a_j > a_{j-1}$.
This implies that for any $0\leq j<i$  there exists $a_j >1$ such that
$w_{a_j} (\pi /2 )=0$ and $w_{a_j}$ has exactly $j$ zeroes in $(0,\pi /2)$. Then by Lemma 3.2 $w_{a_j} '(\pi )=0$ and
$w_{a_j}$ has exactly $2j+1$ zeroes in $(0,\pi )$. This means that we have proved the theorem in case $k$ is odd.

On the other hand since $w_{a_j} $ has exactly one zero less than $w_{a_{j+1}} $ in $(0,\pi /2)$, it follows that
 $w_{a_j}' (\frac{\pi}{2})$ and $w_{a_{j+1}} '(\frac{\pi}{2})$ have different signs. It then follows that there exists $a\in
(a_j , a_{j+1})$ such that $w_a ' (\pi /2 )=0$.  Let $b_j = \inf \{ a \in (a_j , a_{j+1} ) : \ w_a ' (\pi /2)=0 \}$. Then
$w_{b_j}' (\pi /2 )=0$ and for any $a\in (a_j , b_j )$ we have that $w_a '(\pi /2 )$ has the
same sign as $w_{a_j}' (\frac{\pi}{2})$. By the definition of $a_j$ (and the discussion above) we know that for any $a > a_j$
$w_a$ has at least $j+1$ zeroes in $(0,\pi /2 )$ and if moreover  $a$ is close to $a_j$ then
$w_a$ has exactly $j+1$ zeroes. Since the sign of $w_a '(\pi /2 )$ does not change it then follows that
$w_a$ has exactly $j+1$ zeroes for all $a\in (a_j , b_j )$ and Lemma 3.3  implies that $w_{b_j}$ also has exactly $j+1$ zeroes in $(0,\pi /2)$. Then by Lemma 3.1 $w_{b_j}' (\pi )=0$
and $w_{b_j}$ has exactly $2(j+1)$ zeroes in $(0,\pi )$. This proves the theorem when $k$ is even and we have
therefore concluded the proof of the thoerem.

\end{proof}

\section{Bifurcation for positive solutions, proof of Theorem 2.2}

We use bifurcation theory. We denote by $X_I$ the set of invariant functions on  ${\bf S}^n\times {\bf S}^n$.
Consider the Banach space

$$C^{2, \alpha}\left(X_I \right):= X_I \cap C^{2, \alpha} \left({\bf S}^n \times {\bf S}^n\right) .$$

As in Section 2 we identify $C^{2, \alpha}\left(X_I \right)$ with the set of functions $w\in C^{2,\alpha}([0,\pi ])$ such
that $w'(0)=w'(\pi )=0$. Similarly we define $C^{0, \alpha}\left(X_I \right):= X_I \cap C^{0, \alpha} \left({\bf S}^n \times {\bf S}^n\right) $,
which is identified with $C^{0,\alpha}([0,\pi ])$. Now let $C_+^{2, \alpha}\left(X_I \right) $ denote the set of
positive functions on $C^{2, \alpha}\left(X_I \right)$.

We define $S:C_+^{2, \alpha}\left(X_I \right) \times \re_{\geq 0} \rightarrow C^{0, \alpha}\left(X_I \right)$
by

$$S(u,\lambda )= -\Delta_{G_{\delta}} u +\lambda (u -u^{p-1} ).$$

We have for any $\lambda \geq 0$ that $S(1,\lambda )=0$ and we will study solutions of $S(u,\lambda )=0$
which bifurcate for the curve $(1,\lambda )$. The local bifurcation theory is well known, any detail about  what we
will use in this section can be found for instance in Chapter 2 of \cite{Ambrosetti-Malchiodi} or in \cite{Nirenberg}.

Note that

$$S_u ' (1,\lambda) [v] = -\Delta_{G_{\delta } }v - \lambda (p-2) v .$$

Then as in Section 2, we write $v(x)=w(\arccos f(x))$ for a function
$w:[0,\pi ] \rightarrow \re$ with $w'(0)=w'(\pi )=0$, and $S_u ' (1,\lambda) [v]=0$ if and only if

\begin{equation}  \label{ecuacion-eigenvalores}
w''(r)+(n-1)\frac{\cos (r)}{\sin (r)}w'(r) +  \frac{ \lambda (p-2) }{1+ \frac{1}{\delta}}w(r)=0
\end{equation}

We are then led to consider for any positive constant $\beta$ the solution $w_{\beta}$ of the initial
value problem

\begin{equation}  \label{ecuacion-eigenvalores-beta}
w''(r)+(n-1)\frac{\cos (r)}{\sin (r)}w'(r) +  \beta w(r)=0,  \ \ \  w (0)=1, \ \ \ \  \ w'(0)=0.
\end{equation}

This equation is well-known, it corresponds to the eigenvalue equation for the Laplacian on the sphere. If
$\beta_k = k(n+k-1)$ then $w_{\beta_k}$ can be computed explicitly. For instance:

$$ w_{\beta_1}(r) = \cos (r) ,\;\;\;\;  w_{\beta_2} (r) = \frac{n+1}{n} \cos^2 ( r) -\frac{1}{n}, \;\;w_{\beta_3} (r)  = \frac{n+4}{n-2} \cos^3(r) -\frac{6}{n-2} \cos (r) .$$

In general we call

\begin{equation} \label{operador-H-mu-SnxSn}
H_{\beta}(w) :=  w''(r) + \left(n-1\right)\frac{\cos (r)}{\sin (r)} w'(r)  +\beta w(r).
\end{equation}

Then, for each positive integer $k$ we have

$$H_{\beta}(\cos^{k} (r)) = (\beta -\beta_k )\cos^{k} (r)+ k(k-1) \cos^{k-2} (r)  .$$

It is then easy to see that $w_{\beta_k} (r)  = p_k (\cos (r))$,  where $p_k$ is a polynomial  of degree $k$. If
$k$ is odd then $p_k$ is a sum of monomials of odd degree and if $k$ is even $p_k$ is a sum of monomials
of even degree.

For each positive integer $k$, note that $\lambda_k =  \frac{ \beta_k }{p-2}\left( 1+ \frac{1}{\delta}\right)$ and denote by
$L_k = S_u ' (1,\lambda_k ) $.  Then by the previous considerations  $\ker ( L_k )=\langle w_{\beta_k} \rangle $ has dimension 1.

Note that by integration by parts if $x \in C^{2, \alpha} \left(   X_I \right)$ then

$$   0 =\int_0^{\pi}  L_k(w_{\beta_k}) \ x\  dr  = \int_0^{\pi}  L_k(x) \ w_{\beta_k} \  dr .$$
This implies that the range $R(L)$ of $L$, is

$$R(L_k) = \left\{ y \in C^{0, \alpha} \left([0, \pi]\right) :    \int_0^{\pi}  y w_k  \ dr =0 \right\}.$$

On the other hand, note that
$$ S_{u, \lambda} '' (1,\lambda_k)[w_{\beta_k}] =  (p-2)w_{\beta_k}, $$

\noindent and since $ \int_0^{\pi}   w_k^2 \ dr  \neq 0  $  we have that

 $$S_{u, \lambda} '' (1,\lambda_k)[w_k] \notin R(L_k) $$

Therefore,  by (\cite[Theorem 2.8]{Ambrosetti-Malchiodi}), the points
  $(1, \lambda_k )$ are bifurcation points of $S(u, \lambda ) =0$. Moreover, close to $(1, \lambda_k )$
the space of solutions consists of two curves: one is the curve of {\it trivial} solutions $\lambda \mapsto (1, \lambda )$,
and the other one is a curve of nontrivial solutions which  has the form $t \mapsto (u(t) ,  \lambda (t) )$ where $\lambda (0) = \lambda_k$ and
$u(t) = 1 + t w_{\beta_k} + o(t^2 )$.

Note that if $u$ is a nontrivial solution then for any $r\in (0,\pi )$, if $u(r)=1$ then $u' (r) \neq 0$. It follows that the
number of zeroes of $u-1$ is constant in an open $C^2$-neighborhood of $u$.

Now, we  point out that the solution   $w_{\beta_k}$, which  generates $ \ker L_k$,   has exactly $k$ zeroes in  $(0, \pi)$. This is explicitly proved for instance in \cite{YanYan} and in \cite{Petean2}, we just give a sketch of a proof for
completeness:

First note that since $p_k$ is a polynomial of degree $k$ then  $w_{\beta_k}$ can have at most $k$
zeroes in $(0,\pi )$ and they must be simple since $w_{\beta_k}$ solves a second order linear ordinary
differential equation. Then by  direct  calculation, we can verify that $w_{\beta_i}$, have $i$ zeroes in $(0, \pi)$, for $i=1, 2, 3$.
Let $m<l$ be positive integers. Since  $\beta_m = m (n+m-1)< l(n+l-1) =\beta_l$ then, by the Sturm comparison theorem
(see for instance \cite[Page 229]{Ince}), between any two zeroes of $w_{\beta_m}$
there is at least one zero of $w_{\beta_l}$. Hence,  $w_{\beta_l}$ has at least the same number of zeroes as  $w_{\beta_m}$ and if it has exactly the same number,
then   $w_{\beta_l}$ and $w_{\beta_m}$ have  the same sign after the last zero.
Suppose that  $w_k$ has  $k$ zeroes in  $(0, \pi) $. Notice that if  $k$ is even, then  $w_{\beta_k}$ is a  polynomial in  $\cos (r)$ whose  exponents  are even,
therefore $w_{\beta_k}$ is symmetric with  respect to  $\frac{\pi}{2}$  which implies  $w_{\beta_k} (\pi) = w_{\beta_k}(0)=1$.
If $k$ is odd, then  $w_{\beta_k}$ is a  polynomial in $\cos (r)$ whose  exponents  are odd, then $w_{\beta_k}$
is antisymmetric with respect to  $\frac{\pi}{2}$  and we have   $w_{\beta_k} (\pi) =-1$.
By the previous comment, when we move from  $k$ to $k+1$, the corresponding solutions change  sign
 in  $\pi$,  it follows that $w_{\beta_{k+1}}$ must have at least one more zero than   $w_{\beta_k}$. Therefore, by induction
$w_{\beta_{k+1}}$ has at least  $k+1$ zeroes in $(0, \pi)$. By the previous comments  we conclude that
for any positive integer $k$, $w_{\beta_k}$ has
exactly $k$ zeroes in $(0, \pi)$ which are simple, as claimed.

\vspace{.5cm}

Let  $C$ be the closure of the family of positive nontrivial solutions  $(u,\lambda )$ of
$S(u,\lambda )=0$ in $C^{2,\alpha} (X_I  )$. Let $C_k$ be the connected component of  $C$ containing the bifurcation point  $(1,\lambda_k )$.
For the curve of nontrivial solutions $(u(t), \lambda (t))$ close to $(1,\lambda_k )$ we have
$u(t) = 1 + t w_{\beta_k} + o(t^2 )$.
It follows then from the previous comments that  $(u(t) ,\lambda (t) ) \in C_k$ and if $u\neq1$ then
$u$ has exactly $k$-zeroes in $(0,\pi )$. In particular $(1,\lambda_i )$ does not belong to $C_k$ if $i\neq k$.

\vspace{.3cm}

{\it Claim 1}: $C_k$ is not compact
\begin{proof} To prove the claim we will use  the  global bifurcation theorem of Rabinowitz, as in (for instance)
\cite[Theorem 4.8]{Ambrosetti-Malchiodi}. Let us briefly describe how to write the equation in the setup of
the global bifurcation theorem.

Given a solution  $u \colon {\bf S}^n\times {\bf S}^n \to \re_{>0}$ of equation (\ref{positive})
we let  $ w = u-1$.  Then, $u$ is a solution of  (\ref{positive}) if and only if $w$ verifies
\begin{equation}
-\Delta_{G_{\delta}} w  + \lambda (w+ 1) = \lambda (w+1)^{p-1} \label{ecuacion-w+1}.
\end{equation}

Let  $K \colon C^{2, \alpha}\left(X_I \right) \to C^{2, \alpha}\left(X_I\right)$ be the inverse operator of

$$-\Delta_{G_{\delta}} +Id \colon C^{4, \alpha}\left(X_I\right)  \to C^{2, \alpha}\left(X_I \right).$$

The  operator  $K$ is linear and compact.
Consider  the region

$$D:= \left\{(w, \eta) \in C^{2, \alpha}\left(X_I \right)\times \re :w >-1,  \eta > 1 \right\}, $$

\noindent
and define $T \colon D \to C^{2, \alpha}\left(X_I \right)$ by

$$T(w, \eta) = \frac{\eta -1}{p-2} K\left( (w+1)^{p-1} -(p-1)w-1)      \right ).$$

$T$ is  a compact operator and for each $\eta >1$ ,  $ T(0, \eta) =0, \;\;\; T_w'(0, \eta) =0 $. Now  define  $ F\colon D \to C^{2, \alpha}\left(X_I \right) $ by

$$F(w, \eta) = w- \eta K(w) -T(w, \eta)$$

Note that  $F(0, \eta) =0 $ for each $\eta$. And  if we apply  $-\Delta_{G_{\delta} } +Id$  to the equation $F(w, \eta) =0$, we see that $F(w, \eta) =0$ if and only if

$$-\Delta_{G_{\delta}  } w -\frac{\eta -1}{p-2} \left((w+1)^{p-1} - (w+1)\right)=0$$

Therefore, $F(w, \eta) =0$ if and only if $w$ is a solution of equation (\ref{ecuacion-w+1}) for
$\lambda = \frac{\eta -1}{p-2}$.

Let $\eta_k = \lambda_k (p-2) +1$.
Similarly as before we let
$B$ be the closure of the non-trivial solutions  $(w,\eta)$ of
$F(w,\eta)=0$ in $D$ and  $B_k$ be the connected component of  $B$ containing the bifurcation point
$(0,\eta_k )$. In this context we can apply the
 global bifurcation theorem of Rabinowitz
(\cite{Ambrosetti-Malchiodi}, Theorem 4.8): it follows  that either  $B_k$ is not compact or   $B_k$
contains another bifurcation point $(0,\eta_j)$ with $j\neq k$. But we have seen that the second condition does
not hold, therefore $B_k$  is not compact. But $C_k = \{ (w+1 , \frac{\eta- 1}{p-2} ) : (w,\eta ) \in B_k \}$ and
therefore $C_k$ is not compact.

\end{proof}

Note  that  $({\bf S}^n \times {\bf S}^n, G_{\delta})$ has positive Ricci curvature and by \cite[Theorem 6.1]{Bidaut-Veron},  there exist  $\rho>0$
such that  if $\lambda < \rho$ the equation  (\ref{positive}) only has the trivial solution.

\vspace{.5cm}

{\it Claim 2}: For any $\lambda_0$, $0< \rho < \lambda_0$, the set $$A:=\{ (u,\lambda ): S(u, \lambda )=0 , \  \lambda \in [\rho , \lambda_0 ] \}, $$
 is compact.

The claim is well-known, see for instance the proof of \cite[Lemma 2.2]{YanYan}. First one has to note that there exists $\Lambda >0$ such that if $(u,\lambda ) \in A$ then
$u \leq \Lambda$. This is proved by the blow up technique (see for instance the proof in \cite[Theorem 2.1, page 200]{Schoen-Yau}): if there exists a sequence $(u_i  , \lambda_i ) \in A$ and
$x_i \in {\bf S}^n \times {\bf S}^n$ such that
$u_i (x_i ) \rightarrow \infty$ then by taking a subsequence we can assume that $x_i \rightarrow x \in {\bf S}^n \times {\bf S}^n$
and $\lambda_i \rightarrow \lambda \in [\rho , \lambda_0 ]$. Then by taking a normal neighborhood of $x$ and renormalizing $u_i$ one would construct as a limit a positive solution of $\Delta u +\lambda u^{p-1} =0$
in $\re^{2n}$. But since
$p<p_{2n}$ is subcritical such solution does not exist by \cite{Gidas-Spruck}.  Then we consider again the
compact operator  $K \colon C^{2, \alpha}\left(X_I \right) \to C^{2, \alpha}\left(X_I\right)$ from the proof
of Claim 1 (the inverse of $-\Delta_{G_{\delta}} +Id$) and point out that $S(u,\lambda )=0$ if and only if
$u=K(\lambda u^{p-1} -(\lambda -1) u)$: this implies that $A$ is compact.

\vspace{.5cm}

If there exists $\lambda_* > \lambda_k$ such that it does not exist $u \neq 0$ such that  $(u, \lambda_* ) \in C_k$,
then since $C_k$ is connected we have that $C_k \subset C_+^{2,\alpha} (X_I ) \times [\rho ,\lambda_* ]$. But then Claim 2 would imply that $C_k$ is
compact, contradicting Claim 1. Then for any $\lambda > \lambda_k$ there exist $u \neq 0$ such that  $(u, \lambda ) \in C_k$. Since $C_k \cap C_j =\emptyset$ if $j \neq k$, this proves Theorem 2.2.

\end{document}